\documentclass[12pt]{amsart}
\usepackage{amscd}
\usepackage{amsmath}
\usepackage{verbatim}

\usepackage{ytableau}

\usepackage[all, cmtip,arrow]{xy}

\usepackage{pb-diagram,pb-xy}
\usepackage{color}

\usepackage{amssymb}

\numberwithin{equation}{section}

\newtheorem{thm}[equation]{Theorem}
\newtheorem{prop}[equation]{Proposition}

\newtheorem{lemma}[equation]{Lemma}
\newtheorem{cor}[equation]{Corollary}

\theoremstyle{definition}
\newtheorem{rem}[equation]{Remark}
\newtheorem{example}[equation]{Example}
\newtheorem{dfn}[equation]{Definition}

\newcommand{\SB}{\mathop{\mathrm{SB}}}

\renewcommand{\Im}{\mathop{\mathrm{Im}}}

\newcommand{\ind}{\mathop{\mathrm{ind}}}

\newcommand{\CH}{\mathop{\mathrm{CH}}\nolimits}

\newcommand{\GL}{\operatorname{\mathrm{GL}}}

\newcommand{\SL}{\operatorname{\mathrm{SL}}}

\newcommand{\Ch}{\mathop{\mathrm{Ch}}\nolimits}

\newcommand{\Z}{\mathbb{Z}}

\newcommand{\Spec}{\operatorname{Spec}}
\newcommand{\End}{\operatorname{End}}

\newcommand{\Ker}{\operatorname{Ker}}

\renewcommand{\phi}{\varphi}

\usepackage[hypertex]{hyperref}

\title
[Integral motives and Krull-Schmidt]
{Integral motives, relative Krull-Schmidt principle, and Maranda-type theorems}


\author
[N. Semenov]
{Nikita Semenov}

\author
[M. Zhykhovich]
{Maksim Zhykhovich}

\address{Semenov:
Institut f\"ur Mathematik, Johannes Gutenberg-Universit\"at
Mainz, Staudingerweg 9, D-55128, Mainz, Germany}
\email{semenov@uni-mainz.de}

\address{Zhykhovich:
Institut f\"ur Mathematik, Johannes Gutenberg-Universit\"at
Mainz, Stau\-din\-ger\-weg 9, D-55128, Mainz, Germany}
\email{zhykhovi@uni-mainz.de}

\thanks{The authors gratefully acknowledge the support of the Sonderforschungsbereich/Transregio 45 ``Periods,
moduli spaces, and arithmetic of algebraic varieties'' (Bonn-Essen-Mainz).}

\date
{\today}

\begin{document}

\begin{abstract}
In the present article we investigate properties of the category of the integral Grothendieck-Chow motives over a field.

We discuss the Krull-Schmidt principle for integral motives, provide a complete
list of the generalized Severi-Brauer varieties with indecomposable integral motive, and
exploit a relation between the category of motives of twisted flag varieties and integral $p$-adic
representations of finite groups.
\end{abstract}

\maketitle


\section{Introduction}
In the present article we investigate properties of the category of the {\it integral} Gro\-then\-dieck-Chow motives over a field.

A short overview of recent results proved using motivic methods is given e.g. in the introduction of \cite{shells}.

A wide literature is devoted to the problem of finding of integral motivic decompositions via lifting decompositions
given modulo some integers. For example,
Haution and Vishik study liftings of motivic decompositions of smooth projective quadrics from
$\Z/2\Z$-coefficients to $\Z$-coefficients (\cite{Hau}, \cite{vish-lens}). De Clercq studies liftings of
motivic decompositions of twisted flag varieties
from $\Z/p\Z$-coefficients to $\mathbb{F}_{p^n}$-coefficients in \cite{Chls} ($p$ is a prime number). Vishik and Yagita prove some lifting results
in \cite[Section~2]{ViYa}. Finally, Petrov, Semenov, and Zainoulline provide a lifting criterion from $\Z/m\Z$-coefficients to $\Z$-coefficients
for twisted flag varieties of inner type in \cite[Thm.~2.16]{PSZ} for an integer $m$.

Another direction of research involves the Krull-Schmidt principle in the category of Chow motives.
In \cite{Ch-M} Chernousov and Merkurjev proved the Krull-Schmidt
principle for the motives of twisted flag varieties satisfying certain conditions. In particular, they proved that the motive with
$\Z_{(p)}$- or $\Z_p$-coefficients ($p$ is a prime number) of any twisted flag variety of a simple group uniquely decomposes into a direct sum
of indecomposable motives (up to isomorphism and permutation of the factors).

Vishik proved in \cite{vish-lens} that the Krull-Schmidt principle holds for the integral motives of projective
quadrics. However, counterexamples \cite[Example 9.4]{Ch-M} and \cite[Corollary 2.7]{CPSZ} provide
projective homogeneous varieties for which the integral complete motivic decomposition is
not unique.

The following question was raised by experts: To what extent does the Krull-Schmidt principle fail for the integral motives of
twisted flag varieties? For example, is it true that two complete motivic decompositions of a twisted
flag variety become the same over an algebraic closure of the base field?
If this holds, we say that the motivic decompositions are relatively equivalent (see Definition~\ref{defkrull} below)
and if any two complete decompositions of our variety are relatively equivalent, we say that
the relative Krull-Schmidt principle holds.

Note that the motivic decompositions in the counterexamples mentioned above are relatively equivalent and therefore
do not provide counterexamples to the relative Krull-Schmidt principle.

Section~\ref{seckrull} of the present article is devoted to this problem, where we
provide examples and counterexamples. Our proofs rely on our Theorem~\ref{theorem1} proved in Section~\ref{sec3},
which can be interpreted as a classification of all integral motivic decompositions
modulo relative equivalence in terms of reductions modulo primes. The class of varieties, which we consider, consists
of smooth projective varieties possessing a finite Galois splitting field and for which the Rost nilpotence principle
holds. In particular, this includes the twisted flag varieties.

We remark that we do not see, how to prove our Theorem~\ref{theorem1}
referring to the existing literature and combining arguments there. The most close article is \cite{PSZ} and
Theorem~2.16 there, but unfortunately, its proof contains parts (e.g. the proof of the key Proposition~2.15),
which are not plausible to us (including to one of the authors of \cite{PSZ}). Moreover, \cite[Thm.~2.16]{PSZ} contains the condition that
the motivic decomposition under consideration should be {\it $\Z/m\Z$-free} (in the terminology of \cite{PSZ}), and this condition
is not easy to check in practice (cf. Remark~\ref{rem1}).
In the present article we provide a self-contained alternative approach.

The last section of the article is devoted to integral motivic decompositions of the generalized Severi-Brauer
varieties. This problem has a long history, starting probably in 1995,
when Karpenko proved that the integral motive of the classical Severi-Brauer variety of
a division algebra is indecomposable (see \cite{K-i}).
After that many works (let us mention here
\cite{upper}, \cite{CPSZ}, \cite{Chls}, \cite{Zh}, and \cite{Zh2}) were dedicated
to the study of motivic decompositions of the generalized Severi-Brauer varieties.
In particular, Karpenko proved that the integral motive of $\SB_2(A)$ is indecomposable,
if $A$ is a division algebra with $2$-primary index.

However, till this moment there was no complete answer to the question, when
the integral motive of a generalized Severi-Brauer variety is indecomposable.
In the last section of this article we show that the above examples give the exact list of Severi-Brauer
varieties with this property.

Section~\ref{secmaranda} is devoted to the motives with coefficients
in a discrete valuation ring. We formulate all results in this section over the localization $\Z_{(p)}$ of $\Z$
at a prime $p$ or over the $p$-adic integers $\Z_p$. On the other hand, the most results
of this section hold over any discrete valuation ring (sometimes complete). To avoid too technical
exposition we decided to restrict ourselves to these two rings.

The interest on motives with $\Z_{(p)}$- or $\Z_p$-coefficients stems in part from the proof of the
Bloch-Kato conjecture by Rost and Voevodsky, where the fact that the ring of coefficients is flat
over $\Z$ is essential.

Finally, we exploit a relation between the motives of twisted flag varieties and
the category of representations on lattices over d.v.r. This relation allows us to use results from
the representation theory to prove motivic results, as well as use motivic results to give geometric
proofs for results in the representation theory. We illustrate this principle in
Prop.~\ref{maranda}, and Prop.~\ref{conlon} (Maranda and Conlon theorems). We remarks that the relation
between motives and representations appears already in \cite{Ch-M}.

\medskip

{\bf Acknowledgements.}
We would like to thank sincerely Skip Garibaldi and Esther Beneish
for numerous e-mail conversations on the subject of the paper and for sharing with us their knowledge of algebraic
groups and of integral representations of finite groups.

\section{Category of Chow motives}

Let $F$ be a field.
In the present article we work in the category of the Grothendieck-Chow motives over $F$ with coefficients
in a commutative unital ring $\Lambda$ as defined in \cite{EKM}.

If $\Lambda=\Z$, then we speak about {\it integral motives}.
For a motive $M$ over $F$ and a field extension $E/F$ we denote by $M_E$ the extension of scalars.

We denote the Tate motive with twist $n$ by $\Lambda(n)$. A motive $M$ is called
{\it split} (resp. {\it geometrically split}), if it is isomorphic to a finite direct sum of Tate motives
(resp. if $M_E$ is split over some field extension $E/F$).

For a smooth projective variety $X$ over $F$ we denote by $M(X)$ the motive of $X$ in the category
of Chow motives. Every motive $M$ is determined by a smooth projective variety $X$ and a projector
$\pi\in\CH(X\times X)$, where $\CH$ stands for the Chow ring of $X$ modulo rational equivalence.
A motivic decomposition of a motive $M$ is a decomposition into a direct sum. A motivic decomposition
$M=\bigoplus M_i$ is called {\it complete}, if all motives $M_i$ are indecomposable.

Motivic decompositions of a motive $M=(X,\pi)$ correspond to decompositions of the projector $\pi$
into a sum of (pairwise) orthogonal projectors. For a field extension $E/F$ and a variety $Y$ we
call a cycle $\rho\in\CH(Y_E)$ {\it rational}, if it is defined over $F$, i.e. lies in the image of the
restriction homomorphism $\CH(Y)\to\CH(Y_E)$. We say that $\rho$ is $F$-rational, if we want to stress $F$.

We say that the Rost nilpotence principle holds for $X$, if the kernel of the restriction homomorphism
$\End(M(X))\to\End(M(X_E))$ consists of nilpotent elements for all field extensions $E/F$.
At the present moment the Rost nilpotence
principle is proven for motives of twisted flag varieties \cite{CGM}, for motives of varieties in dimensions less
than $3$ (\cite{Gi10} and \cite{Gi12}) and for generically split motives (see \cite{ViZa} and cf. \cite{Br03}),
and it is expected to be valid for all smooth projective varieties.

Let now $G$ be a semisimple linear algebraic group over $F$ (see \cite{Borel}, \cite{Springer}, \cite{Inv}).
A {\it twisted flag variety} of $G$ is the variety of parabolic subgroups of $G$ of some fixed type.
If $G$ is of {\it inner type}, we associate with $G$ a set of prime numbers, called {\it torsion primes}. Namely,
we define this set as the union of all torsion primes of all simple components of $G$,
and for a simple $G$ of inner type this set consists of the prime divisors of $n+1$, if $G$ is of type $\mathrm{A}_n$,
equals $\{2,3\}$ for types $\mathrm{F}_4$, $\mathrm{E}_6$, $\mathrm{E}_7$, $\{2,3,5\}$ for type $\mathrm{E}_8$
and $\{2\}$ in all other cases.

For a semisimple algebraic group $G$ we use the notion ``split'' in a different sense. Namely, $G$ is split,
if it contains a maximal torus which is isomorphic over $F$ to a product of the multiplicative groups $\mathbb{G}_m$,
and we call such a torus {\it split}.
If $G$ is split, then the motive of every twisted flag variety under $G$ with coefficients in any ring is a sum
of Tate motives. It is well-known that any torus has a finite Galois splitting field. In particular, any semisimple group
and any twisted flag variety has a finite Galois splitting field.

Finally, by $[1,n]$ we denote the set of natural numbers $i$ such that $1\le i\le n$.

\section{Integral motives of nilsplit varieties}\label{sec3}
Let $F$ be a field and $M$ be an integral geometrically split motive over $F$.
Over a splitting field extension  $E$ of $F$ the motive $M$ becomes isomorphic
to a finite sum of Tate motives. Let $k$ be a positive integer and let
\begin{equation}\label{eqqq}
M_E = \widetilde{M_1} \oplus \ldots \oplus \widetilde{M_k}
\end{equation}
be a motivic decomposition over $E$ (not necessary complete, so each $\widetilde{M_i}$ is isomorphic to a certain finite sum of Tate motives).

Let now $\Lambda$ be a commutative unital ring.
We say that the above decomposition of $M_E$ is $\Lambda$-{\it admissible}, if there exists
a motivic decomposition over $F$ with $\Lambda$-coefficients
$$M \otimes \Lambda = M_1 \oplus \ldots \oplus M_k \,$$
such that $(M_i)_E \simeq \widetilde{M_i} \otimes \Lambda$ for all $1 \leq i \leq k$.
We say simply $m$-{\it admissible}, if $\Lambda = \Z/m\Z$.

We call a direct summand of $M_E$ $p$-admissible for a prime $p$, if it is a direct summand in a $p$-admissible
decomposition of $M_E$.

\begin{example}
Let $p$ be a prime such that $M\otimes\Z/p\Z$ is a direct sum of Tate motives over $F$.
Then every decomposition~\ref{eqqq} of $M_E$ is $p$-admissible.
In particular, this is the case, if $M$ is a direct summand of the motive of a twisted flag variety of inner type
and $p$ is not a torsion prime.
\end{example}

\begin{dfn}
A smooth projective variety $X$ over a field $F$ is called a {\it nilsplit} variety, if there exists a finite Galois
splitting field extension $E$ of $X$ and the Rost nilpotence principle holds for $X$. We say that a nilsplit
variety is {\it of inner type}, if the Galois group of $E/F$ acts trivially on the Chow group of $X$.
\end{dfn}

\begin{thm}\label{theorem1}
Let $X$ be a nilsplit variety of inner type, $E/F$ be a finite Galois splitting field extension of $X$ and
$M$ be a direct summand of the integral motive of $X$ over $F$.

Then a motivic decomposition of $M_E$ is $\Z$-admissible if and only if it is
$p$-admissible for every prime divisor $p$ of $[E:F]$.
\end{thm}

\begin{proof}
Let $p_i$, $1\leq i\leq l$, denote all distinct prime divisors of $m:=[E:F]$.
The $\Z$-admissibility of a decomposition of $M_E$ clearly implies its $p_i$-admissibility for every $1 \leq i \leq l$.
So we prove another direction of the statement.

By assumptions, $m = \prod _{i=1}^l p_i^{\alpha_i}$ for some $\alpha_i\in\mathbb{N}$.

Let $M = (X, \pi)$, where $\pi$ is a projector in $\End M(X)$.
\noindent Let
\begin{equation}
\label{int-dec}
M_E = \widetilde{M_1} \oplus \ldots \oplus \widetilde{M_k}
\end{equation}
\noindent be a motivic decomposition which is $p_i$-admissible for every $1 \leq i \leq l$.
By definition, for every $i\in [1,l]$ we have the following decomposition
\begin{equation}
\label{proj-dec}
\pi_i = \pi_{i1} + \ldots + \pi_{ik} \, ,
\end{equation}
where $\pi_i$ is the reduction of $\pi_E$ modulo $p_i$ and
$\pi_{ij}$, $1 \leq j \leq k$, are rational over $F$ and orthogonal projectors
in $\End(M_E\otimes \Z/ p_i\Z )$ such that $(X_E, \pi_{ij})\simeq \widetilde{M_j} \otimes \Z/p_i\Z$.

The proof consists of 4 steps. We will construct certain projectors with different coefficient rings
($\Z/p_i\Z$, $\Z/p_i^{\alpha_i}\Z$, $\Z/m\Z$ and finally $\Z$).
We need to check that the constructed projectors are rational, orthogonal and the motivic decomposition~\ref{int-dec}
is admissible on each step.

{\bf Step 1.} (From $\Z/p\Z$ to $\Z/p^\alpha\Z$).
We need the following lemma.

\begin{lemma} \label{lemma} Let $f\colon A \twoheadrightarrow B$ be an epimorphism of finite rings. Let $x \in A$ and $y \in B$ be two projectors such that $f(x)=y$. Then a decomposition of $y$ into a sum of orthogonal projectors in $B$ lifts to a decomposition of $x$ into a sum of orthogonal projectors in $A$.
\end{lemma}
\begin{proof}
By induction argument it is enough to consider the case $y = y_1+y_2$, where
$y_1$ and $y_2$ are orthogonal projectors in $B$. Consider a surjective
homomorphism of finite rings $xAx \twoheadrightarrow yBy$ induced by $f$.
Note that $y_1 \in yBy$, so we can lift $y_1$ to some $\widetilde x_1 \in  xAx$.
By the Fitting Lemma
(here we use that $xAx$ is a finite ring), an appropriate power of
$\widetilde x_1$, which we denote by $x_1$, is a projector.  Finally we take $x_2 = x - x_1$. Clearly the projectors $x_1$ and $x_2$ are orthogonal and the sum $x=x_1+x_2$ is a lifting of the sum $y=y_1+y_2$.
\end{proof}

For every $i\in [1,l]$ we apply the above lemma for $A$ and $B$ being
resp. the subrings of $F$-rational cycles in $\End(M_E \otimes \Z/ p^{\alpha_i}_i\Z )$ and in $\End(M_E\otimes \Z/ p_i\Z )$.
It follows that the decomposition \ref{proj-dec} lifts to a decomposition
$$\pi'_i = \pi'_{i1} + \ldots + \pi'_{ik} \,$$ in the ring $\End(M_E \otimes \Z/p^{\alpha_i}_i\Z)$,
where $\pi'_{ij}$, $ 1 \leq j \leq k $, are rational orthogonal projectors and $\pi'_i$ is the reduction of $\pi_E$
modulo $p^{\alpha_i}_i$. Note that, since $\Z/p^{\alpha_i}_i\Z$ is a local ring, the Krull-Schmidt theorem
holds and, thus, for all $i$, $j$ the motive  $(X_E, \pi'_{ij})$ is a sum of Tate motives. Since $(X_E, \pi_{ij})$ is
the reduction of $(X_E, \pi'_{ij})$ modulo $p$ and $(X_E, \pi_{ij}) \simeq \widetilde{M_j}\otimes \Z/p_i\Z $,
we have $(X_E, \pi'_{ij}) \simeq \widetilde{M_j}\otimes \Z/p^{\alpha_i}_i\Z $ for all $i$, $j$.

{\bf Step 2.} (From $\Z/p^\alpha\Z$ to $\Z/m\Z$ --- Chinese Remainder Theorem).

By $\rho$ we denote the reduction of $\pi_E$ modulo $m$. We will now construct a decomposition
$$\rho = \rho_1 + \ldots + \rho_k$$ into a sum of rational orthogonal projectors such that $(X_E, \rho_j) \simeq \widetilde{M_j}\otimes \Z/m\Z$.

Let us fix $j \in [1, k]$. Since for every $i \in [1,l]$, the motive $(X_E, \pi'_{ij})$ is a sum of Tate motives, we can write for some $r_i$
$$\pi'_{ij} = \sum _{u=1}^{r_i} c_{iu}\times d_{iu} \, ,$$
where $c_{iu}$ and $d_{iu}$ are homogeneous elements in $\CH(X_E)\otimes \Z/p^{\alpha_i}_i\Z$ with
$\deg (c_{iu} \cdot d_{iu'}) = \delta_{uu'}$ for all $u$, $u'$, and $\delta$ denotes the Kronecker delta.

Since for every $i \in [1, l]$ we have $(X_E, \pi'_{ij}) \simeq \widetilde{M_j}\otimes \Z/p^{\alpha_i}_i\Z$
and since our decomposition is $p_i^{\alpha_i}$-admissible,
the rank $r_i$ does not depend on $i$ (and we denote this number by $r$) and
for every $i$ we can choose $c_{iu}$ and $d_{iu}$ such that their codimensions in $\CH(X_E)\otimes \Z/p^{\alpha_i}_i\Z$ do not depend on $i$.

It follows from the Chinese Remainder Theorem that for every $u \in [1,r]$ there exist homogeneous elements
$c_u, \, d_u \in \CH(X_E)\otimes \Z/m\Z$ such that $c_u \equiv c_{iu} \mod p^{\alpha_i}_i$
and $d_u \equiv d_{iu} \mod p^{\alpha_i}_i$ for
every $i \in [1, l]$. Clearly $$\rho_j = c_1\times d_1 + \ldots + c_r \times d_r$$ is a rational projector in
$\End(M_E \otimes \Z/ m\Z )$ such that $(X_E, \rho_j) \simeq \widetilde{M_j}\otimes \Z/m\Z $. In this way we construct
$\rho_j$ for every $1 \leq j \leq k$. Note that these projectors are orthogonal, because their reductions modulo $p^{\alpha_i}_{i}$, $i \in [1  , l]$, are orthogonal.

{\bf Step 3.} (From $\Z/m\Z$ to $\Z$).
Let us denote the graded rings $\CH^*(X_E)$ and $\CH^*(X_E) \otimes \Z/m\Z$ resp. by $V^*$ and $V^*_m$.
Let us lift the decomposition $$\rho = \rho_1 + \ldots + \rho_k$$ to a decomposition of $\pi_E \in \End M(X_E)$ into
a sum of orthogonal projectors. Sorting by codimensions, we can clearly reduce this problem to the case when
$(X_E, \pi_E) \simeq \Z^{\oplus n}(d)$ for some $d, n \in \Z^{\ge 0}$.

We have $\pi_E= \sum_{i=1}^n e_i\times e_i^*$, where $ e_i\in V^d, e^*_i \in V^{\dim X -d}$, $i \in [1,n]$, are
cycles such that $\deg e_i e^*_j = \delta_{ij}$ for all $ 1 \leq i,j \leq n$. For every $j \in [1,k]$, the motive $(X_E, \rho_j)$
is a sum of Tate motives $\Z/m\Z(d)$, so we can decompose $\rho_j$ as a sum of orthogonal projectors of the form
$a\times b$, where $ a\in V_m^d$, $b \in V_m^{\dim X -d}$.
Therefore $\rho = \sum^n_{i=1} a_i\times b_i$, where $\rho_j = \sum_{i \in I_j} a_i\times b_i$ and
$I_1 \sqcup \ldots \sqcup I_k = [1,n]$.

Let us construct now two matrices $A$ and $B$ in $\GL_n(\Z/ m\Z)$. The rows of $A$ (resp. the columns in $B$) are
the coordinates of $a_i$ (resp. of $b_i$), $ 1 \leq i \leq n$, in the basis
$(\bar e_1, \ldots, \bar e_n)$ (resp. $(\bar e^*_1, \ldots, \bar e^*_n)$), where by $x \mapsto \bar{x}$ we mean
the reduction modulo $m$. Since $\deg(a_ib_j) = \delta_{ij}$ for all $1 \leq i,j \leq n$, we have $AB = \mathrm{Id}_n$.
 Replacing $a_1$ by $(\det A) \cdot a_1$ and  $b_1$ by $(\det A)^{-1} \cdot b_1$ we can assume that $A$ and $B$ are in
$\SL_n(\Z/ m\Z)$.

Since $\Z/m\Z$ is a commutative semilocal ring, the group generated by elementary matrices $E_n(\Z/m\Z)$ coincides
with $\SL_n(\Z/m\Z)$. Thus, the reduction homomorphism $\SL_n(\Z)\to\SL_n(\Z/m\Z)$ is surjective and we can lift our
matrices $A$ and $B$ to some matrices $\widetilde{A}$ and $\widetilde{B}$ in $\SL_n(\Z)$
such that $\widetilde A\widetilde B=\mathrm{Id}_n$.

Let $\widetilde{a_i}$ (resp. $ \widetilde{b_i}$), $i \in [1, n]$, be the elements
in $V$ such that their coordinates in the basis $(e_1, \ldots, e_n)$ (resp.
$(e^*_1, \ldots, e^*_n)$) are the $i$-th row of $\widetilde{A}$ (resp. the $i$-th column of $\widetilde{B}$). For every $j \in [1,k]$ we define $\widetilde{\rho_j} = \sum_{i \in I_j} \widetilde{a_i}\times \widetilde{b_i}$, clearly $\widetilde{\rho_j}$ are orthogonal projectors.

By the construction and since $\widetilde B^t\widetilde A^t=\mathrm{Id}_n$,
we have $$\widetilde{\rho_1} + \ldots + \widetilde{\rho_k}=\sum_{i=1}^n e_i\times e_i^*=\pi_E.$$

Since $m$ is the degree of a Galois splitting field of $X$ and $X$ is of inner type,
by transfer argument for any $x \in V \otimes V$ the cycle $ m \cdot x$ is defined over $F$. Thus, the projectors $\widetilde{\rho_j}$, $1 \leq j \leq k$, are $F$-rational.

{\bf Step 4.} (From $E$ to $F$).
Since $X$ satisfies the Rost Nilpotence principle, applying \cite[Ch.~3, Prop.~2.10]{Bass}
to the restriction homomorphism $\End(M)\to\End(M_E)$, we can lift orthogonal
projectors $\widetilde\rho_j$ to orthogonal projectors $\pi_j$, $ 1 \leq j \leq k-1$,
with $(\pi_j)_E = \widetilde{\rho_j}$. Define $\pi_k=\pi-\sum_{j=1}^{k-1}\pi_j$.
Then $\pi_k$ is a projector orthogonal to all $\pi_j$, $1\le j\le k-1$,
$$\pi = \pi_1 + \ldots + \pi_k,$$ and
for every $j \in [1,k]$ we have $(X_E, (\pi_j)_E) \simeq \widetilde{M_j}$.
Thus, the decomposition \ref{int-dec} is $\Z$-admissible.
\end{proof}

\begin{rem}\label{rem1}
Let $p$ and $q$ be two different prime numbers.
Given two motivic decomposition with $\Z/p\Z$- and $\Z/q\Z$-coefficients resp., we can always lift them
to one decomposition with $\Z/pq\Z$-coefficients. On the other hand, not every motivic decomposition with
$\Z/pq\Z$-coefficients can be lifted to a decomposition with $\Z$-coefficients. E.g., the projector
of the form $a\times b+\alpha\cdot c\times d$ with $\deg(ab)=\deg(cd)=1$,
$\deg(ad)=\deg(cb)=0$, and $\alpha$ satisfying two equations $\alpha=0\mod p$ and $\alpha=1\mod q$
cannot be lifted to $\Z$.
\end{rem}

\begin{cor}\label{indecomposable}
In the notation of Theorem~\ref{theorem1} an integral motive $M$ is indecomposable if and only if there exists no non-trivial decomposition of $M_E$ admissible for any
torsion prime $p_i$.
\end{cor}

\begin{rem} As in Remark~\ref{rem1} let $p$ and $q$ be two different prime numbers. Then the Tate motive
with $\Z/pq\Z$-coefficients is decomposable. Namely, for $a$, $b$ with $\deg(a)=1$ one can write
$a\times b=\alpha\cdot a\times b+(1-\alpha)\cdot a\times b$, where $\alpha$ satisfies $\alpha^2=\alpha\mod pq$
and $\alpha\ne 0,1$.

The phenomena of this remark and of remark~\ref{rem1} are related to the fact that over semilocal rings there exist
finitely generated projective modules which are not free.
\end{rem}

\begin{rem}
Let $X$ be a twisted flag variety of inner type. Then by \cite[Section~8]{CGM} $X$ is a nilsplit variety of inner
type. Moreover, in the statement of Thm.~\ref{theorem1} it suffices to consider the torsion primes of $X$ only.
\end{rem}

\section{Maranda-type theorems}\label{secmaranda}

Let $p$ be a prime, $\Z_{(p)}$ be the localization of $\Z$ at $p$, and $\Z_p$ denote the ring of $p$-adic integers.
We begin with general lemmas.

\begin{lemma}\label{padiclemma}
Let $Z$ be a smooth projective variety over a field $F$ and $E$ be a finite Galois field extension of $F$ such that
$\Ch(Z_E):=\CH(Z_E)\otimes \Z_p$ is a free $\Z_p$-module of finite rank. Then the subgroup of $F$-rational cycles is closed in $\Ch(Z_E)$ with respect to the $p$-adic topology.
\end{lemma}
\begin{proof}
Let $(x_{n})_{n \geq 1 }$, $x_{n} \in \Ch(Z_E)$, be a converging sequence of rational cycles with the limit $x \in \Ch(Z_E)$. Let us proof that $x$ is rational. Denote by $G$ the Galois group of $E/F$.

Since the action of $G$ on $\Ch(Z_E)$ is continuous in the $p$-adic topology and $x_n$,
being rational, is $G$-invariant for every $n \geq 1$, the cycle $x$ is $G$-invariant as well.

Let $p^l$ be the maximal power of $p$ dividing $m :=|G|$. Since $(x_{n})_{n \geq 1 }$ converges to $x$, there exists a positive integer $r$ such that
$$x=p^l y + x_r$$
for some $y \in \Ch(Z_E)$. Since by assumptions $\Ch(Z_E)$ is torsion free, $y$ is $G$-invariant.
By transfer argument, $m y$ is rational and, hence, $p^l y$ is rational. Therefore $x$ is $F$-rational.
\end{proof}

The following lemma is \cite[Prop.~7]{Hau}. We remark that Haution formulated this lemma for quadrics,
but the same proof without any change works for any smooth projective variety possessing a finite Galois
splitting field.

\begin{lemma}\label{haution}
Let $R$ denote $\Z_{(p)}$ or $\Z_p$.
Let $X$ be a smooth projective variety over a field $F$ possessing a finite Galois splitting field $E/F$ of degree $m$, and $p^l$
be the largest power of $p$ dividing $m$. A cycle
$\alpha\in\CH(X_E)\otimes R$  is rational if and only if it is invariant under the Galois group of $E/F$ and its reduction
modulo $p^l$ is rational.
\end{lemma}

\begin{thm}\label{padicprop}
Let $X$ be a nilsplit variety over $F$, $p$ be a prime number, $E/F$ be a finite Galois splitting field extension
of the motive of $X$ with $\Z_p$-coefficients and let $M$ be a direct summand of this motive.

Then a motivic decomposition of $M_E$ is $\Z_p$-admissible if and only if it is $p$-admissible.
\end{thm}
\begin{proof}
Let $M = (X, \pi)$, where $\pi$ is a projector in $\End (M(X)\otimes \Z_p)$. For any positive integer $n$ we denote
by $\pi_{E,n}$ the image of $\pi_{E}$ in $\End (M(X_E)\otimes \Z/p^n\Z)$.
Let
\begin{equation}
\label{int-dec2}
M_E = \widetilde{M_1} \oplus \ldots \oplus \widetilde{M_k}
\end{equation}
be a $p$-admissible motivic decomposition of $M_E$.
By definition of the $p$-admissibility, we have a decomposition
\begin{equation}
\label{proj-dec2}
\pi_{E,1} = \pi_{11} + \ldots + \pi_{k1} \, ,
\end{equation}
where $\pi_{i1}$, $1 \leq i \leq k$, are $F$-rational orthogonal projectors
in $\End(M_E\otimes \Z/ p\Z )$ such that $(X_E, \pi_{i1})\simeq \widetilde{M_i} \otimes \Z/p\Z$.

Now we will construct a sequence of decompositions $(P_n)_{n \geq 1 }$

$$ (P_n):  \pi_{E,n} = \pi_{1n} + \ldots + \pi_{kn} \, ,  $$
such that the $n$-th decomposition $(P_n)$ is a decomposition of $\pi_{E,n}$ into a sum of $k$ orthogonal $F$-rational
projectors and for any $n \geq 1$ the decomposition $(P_{n+1})$ is a lifting of $(P_{n})$.

We define $(P_1)$ as the decomposition \ref{proj-dec2}. Using induction on $n$, if $(P_n)$ is constructed, we apply
Lemma~\ref{lemma} and proceeding exactly as in Step 1 of the proof of Theorem~\ref{theorem1} we construct $(P_{n+1})$.

For every $i\in [1,k]$ we define now a sequence $(\rho_{in})_{n \geq 1 }$ of elements in $\End(M_E)$ as follows.
Namely, for $\rho_{in}$ we take an arbitrary $F$-rational $p$-adic representative of $\pi_{in}$ in $\End(M_E)$.
Since $(P_{n+1})$ is a lifting of $(P_{n})$, we have $\rho_{i,n+1} \equiv \rho_{i,n} \mod p^n$ for all positive integers $n$.
Therefore for every $i\in [1,k]$ the sequence $(\rho_{in})_{n \geq 1 }$ converges in the $p$-adic topology to some element $\rho_{i} \in \End(M_E)$. By construction $\rho_i$, $1 \leq i \leq k$, are orthogonal projectors and
\begin{equation}
\label{p-adic-dec}
\pi_E = \rho_{1} + \ldots + \rho_{k} \, .
\end{equation}
By Lemma~\ref{padiclemma} the projectors $\rho_i$, $1 \leq i \leq k$, are rational over $F$.

To finish the proof we lift decomposition \ref{p-adic-dec} (proceeding exactly as in Step 4 of the proof of Theorem~\ref{theorem1}) to a decomposition of $\pi$ into a sum of orthogonal projectors
$$\pi = \pi_{1} + \ldots + \pi_{k} \, . $$
Finally,
for every $i \in [1,k]$ we have $(X_E, (\pi_i)_E) \simeq \widetilde{M_i}$.
Thus, the decomposition \ref{int-dec2} is $\Z_p$-admissible.
\end{proof}

Let now $R$ be a d.v.r. or $R = \Z$. We write $\Ch$ for the Chow group with coefficients in $R$.
Let $X$ and $Y$ be smooth projective varieties such that the $R$-motives of $X$ and $Y$ are split and write $V$ for the $R$-module $\Ch(X)$.
Consider the bilinear form  $$ \mathfrak{b}\colon V \times V \rightarrow R, \quad \mathfrak{b}(x,y)=\mathrm{deg}(x \cdot y)\, .$$

For a correspondence $\alpha\in\Ch(X\times Y)$ we denote by $\alpha_*\colon\Ch(X)\to\Ch(Y)$ its realization (see \cite[\S 62]{EKM}).
By definition $\alpha_*(x)=(\mathrm{pr}_Y)_*(\alpha\cdot (x\times 1))$, $x\in\Ch(X)$, where $(\mathrm{pr}_Y)_*$
is the pushforward of the projection $X\times Y\to Y$.

The following statements are well-known.

\begin{prop}\label{corr-end}
In the above notation for a correspondence $\alpha\in\Ch(X\times X)$
the assignment $\alpha \mapsto \alpha_*$ defines an isomorphism between $R$-algebras $\Ch(X \times X)$ and $\End_{R}(V)$.
\end{prop}

Let $H$ be a finite group which acts $R$-linearly on $V$ and preserves the bilinear form $\mathfrak{b}$.
One can naturally extend this action to $\Ch(X\times X) = V\otimes V$. On the other hand, by Prop.~\ref{corr-end}
$\Ch(X\times X) \simeq \End_{R}(V)$ and this induces an action of $H$ on $\End_{R}(V)$.
Using the fact that $\mathfrak{b}$ is $H$-invariant, one can check that this action coincides with the
natural action of $H$ on $\End_{R}(V)$ (given by $f\mapsto gf$ with $gf(x)=f(gx)$). In particular, we have the following statements:

\begin{cor}\label{cor48}
An element $\alpha \in \Ch(X\times X)$ is $H$-invariant if and only if the endomorphism $\alpha_* \in  \End_{R}(V)$ is $H$-invariant.
More generally, an element $\alpha \in \Ch(X\times Y)$ is $H$-invariant if and only if the homomorphism
$\alpha_*\colon\Ch(X)\to\Ch(Y)$ is $H$-invariant.
\end{cor}

\begin{cor}
\label{projector}
A projector $\pi \in  \End(M(X)\otimes R)$ is $H$-invariant if and only if $\, \Im \alpha_* $ and $\Ker \alpha_*$ are $H$-invariant.
\end{cor}

In general, one cannot replace $\Z_p$ by $\Z_{(p)}$ in the statement of Theorem~\ref{padicprop}
as an example of Esther Beneish shows. Namely, by \cite{Ben} there exists an indecomposable invertible $\Z_{(p)}$-module which is decomposable after
passing to $\Z_p$, hence, using \cite{Ch-M} one can construct a twisted flag variety with a $p$-admissible decomposition
which is not $\Z_{(p)}$-admissible.

Note that the same example of Beneish shows that for motives of twisted flag varieties the Krull-Schmidt principle does not hold with $\Z_{(p)}$-coefficients
(see \cite{Ch-M}). Nevertheless, one can show the following statement.

\begin{prop}\label{maranda}
Let $p$ be a prime number, $M_1$ and $M_2$ be two direct summands of the motives of two nilsplit varieties over
a field $F$ with $\Z_{(p)}$-coefficients, and let $E/F$ be their common finite Galois splitting field.

If $M_1\otimes\Z/p\Z\simeq M_2\otimes\Z/p\Z$, then $M_1\simeq M_2$.
\end{prop}
\begin{proof}
Let $H$ denote the Galois group of $E/F$ acting on $\CH(X_E)$, and $W_1$ and $W_2$
be the realizations of the motives $(M_1)_E$ and $(M_2)_E$ resp.

First note that if $M_1\otimes\Z/p\Z\simeq M_2\otimes\Z/p\Z$, then $M_1\otimes\Z/p^l\Z\simeq M_2\otimes\Z/p^l\Z$
for all $l>1$. Indeed, if $\alpha\colon M_1\to M_2$ and $\beta\colon M_2\to M_1$ are any liftings
of the mutually inverse isomorphisms modulo $p$, then $\alpha \mod p^l$ and $\beta(\alpha\beta)^{p^{l-1}-1} \mod p^l$
are mutually inverse isomorphisms modulo $p^l$.

Applying Cor.~\ref{cor48} we obtain an isomorphism of $(\Z/p^l\Z)[H]$-modules
$W_1\otimes\Z/p^l\Z$ and $W_2\otimes\Z/p^l\Z$. If $l$ is sufficiently large, by the Maranda theorem
\cite[Theorem~30.14]{CuRe} we can lift this isomorphism to an isomorphism of $\Z_{(p)}[H]$-modules $W_1$ and $W_2$.

Applying now Prop.~\ref{corr-end} and Cor.~\ref{cor48} in the opposite direction we obtain an isomorphism between $(M_1)_E$ and $(M_2)_E$.
By construction and by Lemma~\ref{haution} this isomorphism is rational. It remains to apply the Rost
nilpotence principle.
\end{proof}

Finally, using our method we can show the following proposition. In classical terms it follows from the
Conlon theorem.

\begin{prop}\label{conlon}
Let $H$ be a finite group and let $N_1$ and $N_2$ be two invertible $\Z_{(p)}[H]$-modules.
If $N_1\otimes \Z/p\Z\simeq N_2\otimes \Z/p\Z$, then $N_1\simeq N_2$.
\end{prop}
\begin{proof}
This proposition follows from \cite[Prop.~81.17 and Prop.~30.17]{CuRe}.

On the other hand, there is the following geometric proof of this statement. Namely, by
\cite{Ch-M} the category of invertible modules can be embedded in the category of
Chow motives of twisted flag varieties over some field.

Then the statement follows from Prop.~\ref{maranda} (Maranda theorem).
\end{proof}

\section{Relative Krull-Schmidt principle}\label{seckrull}

The following definition was suggested by Charles De Clercq.

\begin{dfn}\label{defkrull}
Let $M(X)$ be the integral motive of a smooth projective variety $X$ over a field $F$.  We say that two complete decompositions
\begin{equation}\label{form2}
M(X) \simeq M_1 \oplus \ldots \oplus M_k \simeq N_1 \oplus \ldots\oplus N_l \, ,
\end{equation}
are \textit{relatively equivalent}, if $k=l$ and there exists a permutation $\sigma$ of $\{1, \ldots, k\} $ such that
for any $i \in \{1, \ldots, k\}$, we have $(M_i)_{\bar{F}} \simeq (N_{\sigma(i)})_{\bar{F}}$.
We say that $M(X)$ satisfies the \textit{relative Krull-Schmidt principle} if all complete decompositions of $M(X)$
are relatively equivalent.
\end{dfn}

The following proposition shows
that the relative Krull-Schmidt principle holds for a wide class of projective homogeneous varieties.

\begin{prop} Let $X$ be a twisted flag variety of inner type admitting a splitting field of $p$-primary degree for some prime number $p$. Then the relative Krull-Schmidt principle holds for the integral motive of $X$.
\end{prop}
\begin{proof}
Let $M(X)$ be the integral motive of $X$. Consider two complete integral motivic decomposition of $M(X)$ as
in formula~\ref{form2} and the reductions of these motivic decompositions modulo $p$.

By Theorem~\ref{theorem1}, since $X$ has a splitting field of a $p$-primary degree, any indecomposable
summand of $M(X)$ remains indecomposable modulo $p$. Therefore after reduction modulo $p$ any complete decomposition of $M(X)$
remains complete.

By the Krull-Schmidt principle for $\Z/p\Z$-coefficients the motive $M(X) \otimes \Z/ p\Z$ has a unique complete
decomposition (up to isomorphism and permutation of the factors).

Finally, for any two geometrically split integral motives $M$ and $N$ over $F$ we obviously have
$$ M_{\bar{F}} \otimes \Z/ p\Z \simeq N_{\bar{F}} \otimes \Z/ p\Z  \quad  \Longrightarrow  \quad  M_{\bar{F}} \simeq N_{\bar{F}} \, .$$
Therefore it follows that any two complete decompositions of the integral motive of $X$ are relatively equivalent.
\end{proof}

Now using Theorem~\ref{theorem1} we provide an example of a twisted flag variety for which the relative Krull-Schmidt
principle fails.

Let $G$ be a simple algebraic group of type $\mathrm{F}_4$ over $F$ and $X$ the projective $G$-homogeneous
variety of maximal parabolic subgroups of $G$ of type $1$ (the enumeration of simple roots follows Bourbaki).
The torsion primes for $G$ are $2$ and $3$ and any group $G$ of type $\mathrm{F}_4$ has a splitting field $E$ of degree $6$.
Assume that our group $G$ does not have splitting fields of degree $2$ and $3$. Then the motivic
decompositions of $X$ modulo $2$ and modulo $3$ are known (see \cite{NSZ} and \cite[Section~7]{PSZ}).
Namely, we have over $F$
$$M(X)\otimes\Z/3\Z=\bigoplus_{i=0}^7R_3(i)\text{ and }
M(X)\otimes\Z/2\Z=\oplus_{i\in\{0,1,2,4,5,7,8,10,11,12\}}R_2(i)\bigoplus R_2^{\oplus 2}(6),$$
where $R_2$ and $R_3$ are indecomposable motives and $(R_2)_E\simeq\Z/2\Z\oplus\Z/2\Z(3)$
and $(R_3)_E\simeq\Z/3\Z\oplus\Z/3\Z(4)\oplus\Z/3\Z(8)$. We remark that the motives
$R_2$ and $R_3$ are the {\it Rost motives} corresponding to the cohomological
invariants $f_3$ and $g_3$ resp. (see \cite[\S40]{Inv}).

We represent the Tate motives in a motivic decomposition over $E$ graphically as {\it boxes}. We draw them from left to right according to
their shifts. For example,
\ytableausetup{smalltableaux}
\begin{ytableau}
\empty &  \empty\\
\none & \empty
\end{ytableau}\,
means $\Lambda\oplus\Lambda(1)^{\oplus 2}$, where $\Lambda$ is the coefficient ring.

We put letters in the boxes to collect Tate motives belonging to the same indecomposable motive over the
base field $F$. For example,
\ytableausetup{smalltableaux}
\begin{ytableau}
\text{A} & \text{B}\\
\none & \text{B}
\end{ytableau}\,
means that over $F$ the decomposition is $M_1\oplus M_2$ with $M_1$ and $M_2$ indecomposable and with
$(M_1)_E\simeq\Lambda$ and $(M_2)_E\simeq\Lambda(1)^{\oplus 2}$.

Drawing the decompositions of $M(X)$ with $\Z/2\Z$-, $\Z/3\Z$-coefficients and (applying Theorem~\ref{theorem1})
with $\Z$-coefficients we get:

$$\begin{tabular}{cc}

{\tiny mod 3}&{\tiny mod 3}\\
\ytableausetup{smalltableaux}
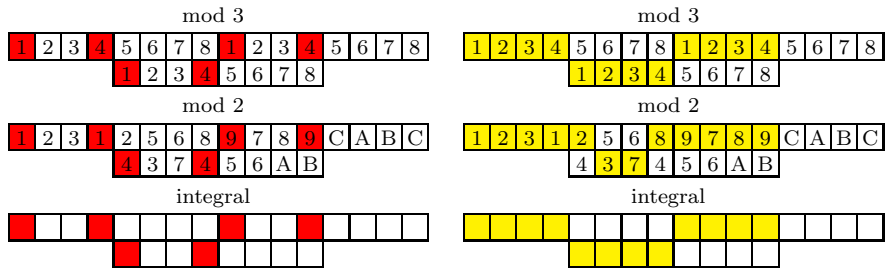
\begin{ytableau}
*(red) 1&  2 & 3&*(red)4&5&6&7&8&*(red)1&2&3&*(red)4&5&6&7&8\\
\none & \none &\none &\none &*(red)1 &2&3&*(red)4&5&6&7&8
\end{ytableau}
&
\ytableausetup{smalltableaux}
\begin{ytableau}
*(yellow) 1& *(yellow)2 & *(yellow)3&*(yellow)4&5&6&7&8&*(yellow)1&*(yellow)2&*(yellow)3&*(yellow)4&5&6&7&8\\
\none & \none &\none &\none &*(yellow)1 &*(yellow)2&*(yellow)3&*(yellow)4&5&6&7&8
\end{ytableau}
\\
{\tiny mod 2} & {\tiny mod 2}\\
\ytableausetup{smalltableaux}
\begin{ytableau}
*(red) 1&2&3&*(red)1&2&5&6&8&*(red)9&7&8&*(red)9&\text{C}&\text{A}&\text{B}&\text{C}\\
\none & \none &\none &\none & *(red) 4&3&7&*(red)4&5&6&\text{A}&\text{B}
\end{ytableau}
&
\ytableausetup{smalltableaux}
\begin{ytableau}
*(yellow) 1&*(yellow)2&*(yellow)3&*(yellow)1&*(yellow)2&5&6&*(yellow)8&*(yellow)9&*(yellow)7&*(yellow)8&*(yellow)9&\text{C}&\text{A}&\text{B}&\text{C}\\
\none & \none &\none &\none & 4&*(yellow)3&*(yellow)7&4&5&6&\text{A}&\text{B}
\end{ytableau}
\\
{\tiny integral} & {\tiny integral}\\
\ytableausetup{smalltableaux}
\begin{ytableau}
*(red) &&&*(red)&&&&&*(red)&&&*(red)&&&&\\
\none & \none &\none &\none & *(red) &&&*(red)&&&&
\end{ytableau}
&
\ytableausetup{smalltableaux}
\begin{ytableau}
*(yellow) &*(yellow)&*(yellow)&*(yellow)&&&&&*(yellow)&*(yellow)&*(yellow)&*(yellow)&&&&\\
\none & \none &\none &\none &*(yellow) &*(yellow)&*(yellow)&*(yellow)&&&&
\end{ytableau}
\end{tabular}
$$

The colored boxes represent the Tate motives we combine together to lift decompositions modulo torsion primes to
integral decompositions. Thus, combining the Tate motives in two different ways as shown in these pictures, we get two integral decompositions of $M(X)$, which are not relatively
equivalent.

We can write the above pictures as formulas. Namely, we constructed two integral indecomposable
(by Corollary~\ref{indecomposable}) direct summands $L_1$ (red) and $L_2$ (yellow) of $M(X)$ such that

\begin{tabular}{l|l}
$(L_1)_E\simeq\bigoplus_{i\in\{0,3,4,7,8,11\}}\Z(i)$&
$(L_2)_E\simeq\bigoplus_{i=0}^{11}\Z(i)$\\
$L_1\otimes\Z/3\Z\simeq R_3\oplus R_3(3)$&
$L_2\otimes\Z/3\Z\simeq\bigoplus_{i=0}^3 R_3(i)$\\
$L_1\otimes\Z/2\Z\simeq R_2\oplus R_2(4)\oplus R_2(8)$&
$L_2\otimes\Z/2\Z\simeq \oplus_{i=0}^2 R_2(i)\bigoplus\oplus_{i=6}^8 R_2(i)$.
\end{tabular}

\section{Motivic decomposability of generalized Severi-Brauer varieties}
Let $F$ be a field and let $A$ be a central simple $F$-algebra of degree $n$. We
write $\SB(k,A)$ for the generalized Severi-Brauer variety of
right ideals in $A$ of reduced dimension $k$ for
$k=1,\dots,n$. In particular, $\SB(n,A)=\Spec F$ and $\SB(1,A)$ is the classical Severi-Brauer variety of $A$. The generalized Severi-Brauer varieties are twisted forms of Grassmannians (see \cite[\S I.1.C]{Inv}).

In this section we study integral motivic decomposability of generalized Severi-Brauer varieties.
Since $\SB(k,A) \simeq \SB(n -k, A^{\mathrm{op}})$, the general case can be reduced to the case $1 \leq k \leq (\deg A)/2$.

\begin{thm} Let $A$ be a central simple $F$-algebra and let $k$ be an integer such that $1 \leq k\leq (\deg A)/2$.
The integral motive of the generalized Severi-Brauer variety $\SB(k,A)$ is decomposable except the following two cases:
\begin{enumerate}
  \item [1)] $k=1$ and $A$ is a division algebra (classical Severi-Brauer variety);
  \item [2)] $k=2$ and $A$ is a division algebra with $2$-primary index.
\end{enumerate}

\end{thm}

\begin{proof}
Let $n=\deg A$. By assumption $1 \leq k \leq n/2$.
If $A$ is not a division algebra, then by \cite[Corollary 10.19]{Karpenko-thesis}
the integral motive of $\SB(k,A)$ is decomposable for any $k$. From now on we assume that $A$ is division.

By \cite[Theorem~2.2.1]{K-i} the integral motive of the Severi-Brauer variety $\SB(1,A)$
is indecomposable and so we assume $k>1$.

Let us mention another already studied case: $\deg A$ is $p$-primary for
some prime number $p$. By Corollary~\ref{indecomposable}, the integral motivic
decomposability in this case is equivalent to the motivic decomposability
modulo $p$, which was completely studied in \cite{upper} and \cite{Zh2}.
Namely, for a $p$-primary $A$ the motive of $\SB(k,A)$ is decomposable if and only if $p=2$ and $k=2$.

We can now assume that $n=ml$, where $2 \leq l < m$ are coprime positive integers.
We have $A= A_m \otimes_F A_l$, where $A_l$ and $A_m$ are division algebras and
$\ind A_l =l$, $\ind A_m =m$. We take the following notation: $X=\SB(k,A)$, $Y_l= \SB(1, A_l)$, $Y_m= \SB(1, A_m)$. We denote by $N$ the integral motive $M(Y_l\times Y_m)(d)$, where $d= n-l-k+1$. Let $E$ be a splitting field extension of $F$ for the variety $X$.
By Corollary \ref{indecomposable}, in order to prove our theorem, it suffices to show that for any prime $p$ dividing $n$ the motive $N_E$ is $p$-admissible for $M(X_E)$.

For every prime $p$ dividing $n$ we have
$$
\begin{array}{cl}
M(X)\otimes \Z/p\Z \simeq M(\SB(k, \mathrm{M}_m(A_l))) \otimes \Z/p\Z & \mbox{ if }  p|\,l \\
M(X)\otimes \Z/p\Z \simeq M(\SB(k, \mathrm{M}_l(A_m))) \otimes \Z/p\Z & \mbox{ if }  p|\,m \, .
\end{array}
$$

By \cite[Corollary 10.19]{Karpenko-thesis} (applied to $\SB(k, \mathrm{M}_m(A_l))$ and resp. to $\SB(k, \mathrm{M}_l(A_m))$) we obtain the following motivic summand of $X$ modulo $p$:

$$
\begin{array}{clll}
M(Y_l\times Z_l)(d) & \mbox{ if }  p|\,l\, , & \mbox{where} \,\,\,  Z_l= \SB(k-1, \mathrm{M}_{m-1}(A_l))\,;\\
M(Y_m\times Z_m) (d-m+l) & \mbox{ if }  p|\,m \, , & \mbox{where} \,\, Z_m= \SB(k-1, \mathrm{M}_{l-1}(A_m))\,.
\end{array}
$$
By \cite[Proposition 4.3]{I-K} this summand in the case $p|\,l$ (resp. $p|\,m$) decomposes into a sum of consecutive shifts (without blanks) of $M(Y_l)\otimes \Z/p\Z$ (resp. of $M(Y_m)\otimes \Z/p\Z$).
Therefore in order to prove that for every prime $p$ dividing $n$ the motive $N_E$ (recall that $N =M(Y_l\times Y_m)(d)$) is $p$-admissible for $M(X_E)$ it suffices to check the following elementary inequalities:
$$\dim Z_l \geq \dim Y_m \quad \mbox{and} \quad \dim Z_m -m+l \geq \dim Y_l \, . $$
\noindent We use formulas $\dim Y_l = l-1$, $\dim Y_m = m-1$, $\dim Z_l = (k-1)(l(m-1)-k+1)$, $\dim Z_m = (k-1)(m(l-1)-k+1)$, and assumptions $2 \leq l < m$, $2 \leq k\leq n/2$.
\end{proof}

\end{document}